\documentclass[11pt]{article}
\usepackage[a4paper,margin=2.7cm]{geometry}
\usepackage[T1]{fontenc}
\usepackage[utf8]{inputenc}
\usepackage{lmodern}
\usepackage{hyperref}

\usepackage{amsmath,amssymb,amsfonts}
\usepackage{amsthm}
\usepackage{mathrsfs}

\theoremstyle{plain}
\newtheorem{theorem}{Theorem}[section]
\newtheorem{lemma}[theorem]{Lemma}
\newtheorem{proposition}[theorem]{Proposition}
\newtheorem{corollary}[theorem]{Corollary}
\newtheorem{problem}[theorem]{Open Problem}

\theoremstyle{definition}
\newtheorem{example}[theorem]{Example}
\newtheorem{remark}[theorem]{Remark}

\theoremstyle{definition}
\newtheorem{definition}[theorem]{Definition}

\begin{document}

\title{A General Framework for Operator Continuity between Riesz Spaces}
\author{Ayşe Uyar\thanks{Email: ayseu@gazi.edu.tr}\\Retired Professor of Mathematics, Istanbul, Türkiye}
\date{}

\maketitle

\begin{abstract}
We introduce and study $\mathcal{FG}$-order continuous operators between
Riesz spaces, based on Tantrawan's generalized (unbounded) order
convergence. This framework includes several familiar classes of
operators and allows their properties to be studied in a unified
setting. We establish order boundedness and lattice-theoretic results;
in particular, we show that strongly order continuous operators are
automatically order bounded and prove that the modulus of an order
bounded $\mathcal{FG}$-order continuous operator remains
$\mathcal{FG}$-order continuous under suitable assumptions. The latter
result gives an affirmative answer to a previously posed problem for
strongly order continuous operators. We also establish a general result
ensuring classical order continuity, from which order continuity
results for several familiar classes of operators follow. Furthermore,
we characterize, under suitable hypotheses, when the space of order
bounded $\mathcal{FG}$-order continuous operators is a band. Finally,
we prove an extension theorem for positive $\mathcal{FG}$-order
continuous operators, providing a counterpart of Veksler's extension
theorem in this framework.
\end{abstract}

\noindent\textbf{Keywords:} $\mathcal{FG}$-order continuous operator; order continuous operator; unbounded order continuous operator; strongly order continuous operator.

\medskip
\noindent\textbf{MSC 2020:} Primary 47B60, 47B65; Secondary 46A40.

\section{Introduction and Preliminaries}

Order convergence and its variants play an important role in the
theory of Riesz spaces and in the study of operators acting between
them. Among these notions, unbounded order convergence has received
considerable attention, partly because of its close connection with
classical modes of convergence such as coordinatewise and almost
everywhere convergence. More recently, Tantrawan \cite{Tantrawan25}
introduced a generalized form of (unbounded) order convergence that
provides a common framework encompassing both order and unbounded
order convergence. The purpose of this paper is to study operator
continuity with respect to this generalized convergence.

Throughout the paper, all Riesz spaces are assumed to be Archimedean,
and, unless otherwise stated, an \emph{operator} between Riesz spaces
means a linear operator.
For terminology and standard results on Riesz spaces and operators,
we refer to
\cite{AliprantisBurkinshaw03,
AliprantisBurkinshaw06,LuxemburgZaanen71, Zaanen83}.

We begin by recalling the convergence notions that will be used
throughout the paper.

\begin{definition}\label{def:convergences}
Let $X$ be a Riesz space and let $(x_\alpha)_{\alpha\in\Gamma}$ be
a net in $X$.

\begin{enumerate}
\item[(i)] The net $(x_\alpha)$ is said to be
\emph{strongly order convergent} (or \emph{so-convergent}) to
$x\in X$ if there exists a net $(z_\alpha)_{\alpha\in\Gamma}$ in
$X$ such that $z_\alpha\downarrow0$ and there exists
$\alpha_0\in\Gamma$ such that
\[
|x_\alpha-x|\leq z_\alpha
\qquad\text{for all }\alpha\geq\alpha_0.
\]
In this case, we write
$x_\alpha\xrightarrow{so}x$.

\item[(ii)] The net $(x_\alpha)$ is said to be
\emph{order convergent} (or \emph{o-convergent}) to $x\in X$ if
there exists a net $(z_\beta)_{\beta\in\Lambda}$ in $X$ such that
$z_\beta\downarrow0$ and, for every $\beta\in\Lambda$, there exists
$\alpha_0\in\Gamma$ such that
\[
|x_\alpha-x|\leq z_\beta
\]
for all $\alpha\geq\alpha_0$. In this case, we write
$x_\alpha\xrightarrow{o}x$.

\item[(iii)] The net $(x_\alpha)$ is said to be
\emph{unbounded order convergent} (or \emph{uo-convergent}) to
$x\in X$ if
\[
|x_\alpha-x|\wedge u\xrightarrow{o}0
\qquad\text{for every }u\in X_+.
\]
In this case, we write $x_\alpha\xrightarrow{uo}x$.

\item[(iv)] If $X$ is an $f$-algebra, the net $(x_\alpha)$ is said
to be \emph{multiplicative order convergent}
(or \emph{mo-convergent}) to $x\in X$ if
\[
|x_\alpha-x|u\xrightarrow{o}0
\qquad\text{for every }u\in X_+.
\]
In this case, we write $x_\alpha\xrightarrow{mo}x$.

\item[(v)] Suppose that the order dual $X^\sim$ separates the points
of $X$. The net $(x_\alpha)$ is said to be
\emph{bidual bounded unbounded order convergent}
(or \emph{bbuo-convergent}) to $x\in X$ if
\[
x_\alpha\xrightarrow{uo}x
\]
and the canonical images $(\widehat{x_\alpha})$ are eventually
order bounded in $X^{\sim\sim}$. In this case, we write
$x_\alpha\xrightarrow{bbuo}x$.
\end{enumerate}
\end{definition}

If $X$ is a Dedekind complete Riesz space, then strong order
convergence is equivalent to order convergence. Clearly,
$x_\alpha\xrightarrow{o}x$ also implies
$x_\alpha\xrightarrow{uo}x$. Order convergence and unbounded order convergence are well-established
notions and have been extensively studied in the literature. For order convergence, we refer to \cite{AbramovichSirotkin05}, and
for further background and fundamental results on unbounded order
convergence to \cite{GaoTroitskyXanthos17, GaoXanthos14}.

More recently, multiplicative order convergence and bbuo-convergence
have been introduced and studied as further variants of order
convergence. For multiplicative order convergence, we refer to
\cite{Aydin20,GurkokTuran25}, and for bbuo-convergence to
\cite{AlpayEmelyanovGorokhova21}.

We now turn to the convergence notion that forms the basis of the
present work. The following notion of convergence was introduced by
Tantrawan \cite{Tantrawan25}.

\begin{definition}\label{def:F-order_convergence}
{\rm \cite[Definition~1]{Tantrawan25}}
Let $\mathcal F$ be a family of increasing subadditive maps
$F:X_+\to X_+$ satisfying $F(0)=0$. A net $(x_\alpha)$ in $X$ is
said to \textit{$\mathcal F$-order converge} to $x\in X$, denoted by
\[
x_\alpha\xrightarrow{\mathcal F\text{-}o}x,
\]
if
\[
F(|x_\alpha-x|)\xrightarrow{o}0
\]
for every $F\in\mathcal F$.
\end{definition}

We denote by $\operatorname{IS}(X)$ the collection of all increasing
subadditive maps $F:X_+\to X_+$ satisfying $F(0)=0$. Thus,
$\mathcal F\subseteq\operatorname{IS}(X)$ in
Definition~\ref{def:F-order_convergence}.

Tantrawan showed that suitable choices of $\mathcal F$ recover order
and unbounded order convergence, while other choices may yield
convergences different from both; see \cite[Example~2]{Tantrawan25}.
Moreover, convergence induced by a locally solid topology can also be
represented as $\mathcal F$-order convergence; see
\cite[Example~3]{Tantrawan25}.

Unlike order and unbounded order convergence, $\mathcal F$-order
convergence need not have unique limits. Tantrawan
\cite[Theorem~4]{Tantrawan25} showed that $\mathcal F$-order limits
are unique if and only if
\[
\bigcap_{F\in\mathcal F}\ker F=\{0\}.
\]
We shall refer to this condition as the \emph{kernel separation
property} of $\mathcal F$.
The basic algebraic and lattice properties of $\mathcal F$-order
convergence were established in
\cite[Proposition~5]{Tantrawan25}. Thus, when appropriate, we may assume without loss of generality
that the net is positive and $\mathcal F$-order converges to zero.

We shall use the following distinguished families. Let
\[
\mathcal I_X:=\{I_{X_+}\},
\qquad
\mathcal T_X:=\{T_u:u\in X_+\},
\]
where
\[
T_u(x)=x\wedge u
\qquad (x\in X_+).
\]
If $X$ is a normed Riesz space, fix $0<u\in X$ and set
\[
\mathcal N_X:=\{N_u\},
\qquad
N_u(x)=\|x\|u.
\]
Finally, if $X$ is an $f$-algebra, we put
\[
\mathcal M_X:=\{M_u:u\in X_+\},
\qquad
M_u(x)=xu.
\]

We refer to $\mathcal I_X$, $\mathcal T_X$, $\mathcal N_X$, and
$\mathcal M_X$ as the identity, truncation, norm-induced, and
multiplication families, respectively. Whenever defined, these are
subfamilies of $\operatorname{IS}(X)$.

The families $\mathcal I_X$, $\mathcal T_X$, and $\mathcal N_X$
have the kernel separation property. The same is true for
$\mathcal M_X$ whenever $X$ is a semiprime $f$-algebra.

For the subclasses of $\operatorname{IS}(X)$, we use the notation
$\operatorname{IS}_o(X)$, $\operatorname{IS}_{\sigma o}(X)$,
$\operatorname{IS}_d(X)$, and $\operatorname{IS}_b(X)$ for the
order continuous, $\sigma$-order continuous, disjointness preserving,
and bounded maps, respectively. Here, following
\cite{Tantrawan25}, a map $F\in\operatorname{IS}(X)$ is called
bounded if there exists $z\in X_+$ such that
\[
F(X_+)\subseteq[0,z].
\]
We also write $\operatorname{IS}_{\vee}(X)$ for the class of maps in
$\operatorname{IS}(X)$ that preserve arbitrary existing suprema.
Clearly, $\mathcal I_X\subseteq\operatorname{IS}_{\vee}(X)$.
Moreover, by the infinite distributive law,
$\mathcal T_X\subseteq\operatorname{IS}_{\vee}(X)$, and, if $X$ is
an $f$-algebra, then
$\mathcal M_X\subseteq\operatorname{IS}_{\vee}(X)$ by
\cite[Proposition~2.9]{AydinEmelyanovGorokhova25}.
Membership of the above families in the subclasses introduced here
is generally immediate from their definitions, although some
properties may require additional assumptions; for instance,
$\mathcal N_X\subseteq\operatorname{IS}_o(X)$ whenever $X$ has
order continuous norm.

A central question associated with these convergence notions is the study of
operators that preserve them. Operators preserving $uo$-convergence are called
\emph{unbounded order continuous} ($uo$-continuous) operators. If $X$ is a
normed Riesz space, an operator $T:X\to Y$ is called
\emph{boundedly unbounded order continuous} (buo-continuous) if
$Tx_\alpha\xrightarrow{uo}0$ whenever $(x_\alpha)$ is a norm bounded
$uo$-null net in $X$. Both classes were introduced in
\cite{BahramnezhadAzar18}; see also the correction
\cite{BahramnezhadAzar19}. The class of $uo$-continuous operators was
further investigated in
\cite{TuranAltinGurkok22,TuranGurkok24,Uyar_WhenRemains}.

An operator $T:X\to Y$ is called \emph{strongly order continuous} if
$Tx_\alpha\xrightarrow{o}0$ whenever $x_\alpha\xrightarrow{uo}0$; such
operators were studied in \cite{Bahramnezhad_Azar}. Operators preserving
multiplicative order convergence, called \emph{multiplicative order continuous}
(mo-continuous) operators, were considered in
\cite{AydinEmelyanovGorokhova25,GurkokTuran25}. To the best of our knowledge,
the corresponding operator class for bbuo-convergence has not previously been
studied. We naturally call an operator preserving bbuo-convergence
\emph{bbuo-continuous}.

Motivated by Tantrawan's generalized order convergence, we introduce
$\mathcal{FG}$-order continuous operators, where the families $\mathcal F$
and $\mathcal G$ determine the modes of convergence on the domain and range,
respectively. Suitable choices of these families recover order continuous,
$uo$-continuous, strongly order continuous, norm continuous, and
$mo$-continuous operators, while buo-continuous and bbuo-continuous operators
are also encompassed by this framework. Throughout the paper, the general
results obtained within this framework are specialized to these particular
classes, recovering or extending known results and yielding new properties of
strongly order continuous, mo-continuous, buo-continuous, and bbuo-continuous
operators.

We provide examples showing that $\mathcal{FG}$-order continuity depends
essentially on the choice of the families. Although it does not imply order
boundedness in general, we establish conditions under which it does; in
particular, every strongly order continuous operator is automatically order
bounded.

We next establish a general supremum-envelope theorem, which serves as a key
tool for studying lattice properties of $\mathcal{FG}$-order continuous
operators. As an application, we show, under suitable assumptions, that
$\mathcal{FG}$-order continuity is preserved under taking moduli and that
$L_{\mathcal{FG}}(X,Y)$ is a Riesz subspace of $L_b(X,Y)$. In particular,
this gives an affirmative answer to a previously posed question concerning
the modulus of strongly order continuous operators.

We then investigate the relationship between $\mathcal{FG}$-order continuity
and classical order continuity. We give conditions under which an order
bounded $\mathcal{FG}$-order continuous operator is order continuous and,
through the specializations of the framework, obtain order continuity results
for the corresponding classes of operators. In particular, we extend a known
result for $uo$-continuous operators by removing the Dedekind completeness
assumption on the range space.

We further investigate the order structure of $L_{\mathcal{FG}}(X,Y)$.
Under suitable assumptions, we show that $L_{\mathcal{FG}}(X,Y)$ is an ideal
of $L_b(X,Y)$ and obtain finite-dimensional characterizations in terms of its
band property and its equality with $L_n(X,Y)$. Finally, in the spirit of
Veksler's classical extension theorem, we establish an extension theorem for
positive $\mathcal{FG}$-order continuous operators and characterize
$\mathcal{FG}$-order continuity on ideals.

\section{$\mathcal{FG}$-Order Continuous Operators}

In this section, we introduce $\mathcal{FG}$-order continuous
operators and their sequential counterparts, and study their basic
properties and relationships with several familiar classes of
operators on Riesz spaces.

\begin{definition}\label{def:FG-order_continuous}
Let $X$ and $Y$ be Riesz spaces, and let
$\mathcal F\subseteq\operatorname{IS}(X)$ and
$\mathcal G\subseteq\operatorname{IS}(Y)$. An operator $T:X\to Y$
is said to be
\begin{enumerate}
\item[(i)] \emph{$\mathcal{FG}$-order continuous} if
\[
x_\alpha\xrightarrow{\mathcal F\text{-}o}0
\quad\Longrightarrow\quad
Tx_\alpha\xrightarrow{\mathcal G\text{-}o}0.
\]

\item[(ii)] \emph{$\mathcal{FG}\sigma$-order continuous} if
\[
x_n\xrightarrow{\mathcal F\text{-}o}0
\quad\Longrightarrow\quad
Tx_n\xrightarrow{\mathcal G\text{-}o}0.
\]
\end{enumerate}
\end{definition}

The space of all order bounded $\mathcal{FG}$-order continuous
operators is denoted by $L_{\mathcal{FG}}(X,Y)$; that is,
\[
L_{\mathcal{FG}}(X,Y)
=
\{T\in L_b(X,Y):T\text{ is }\mathcal{FG}\text{-order continuous}\}.
\]
Similarly, $L_{\mathcal{FG}\sigma}(X,Y)$ denotes the space of
all order bounded $\mathcal{FG}\sigma$-order continuous operators
from $X$ to $Y$. Clearly,
\[
L_{\mathcal{FG}}(X,Y)
\subseteq
L_{\mathcal{FG}\sigma}(X,Y),
\]
and both are ordered vector subspaces of $L_b(X,Y)$.
For brevity, $L_{\mathcal{FG}}(X,X)$ and
$L_{\mathcal{FG}\sigma}(X,X)$ will be denoted by
$L_{\mathcal{FG}}(X)$ and $L_{\mathcal{FG}\sigma}(X)$,
respectively.

For a family $\mathcal F\subseteq\operatorname{IS}(X)$, we define
the \emph{$\mathcal F$-order continuous dual} of $X$ by
\[
X^\sim_{\mathcal F}
=
\left\{
f\in X^\sim:
x_\alpha\xrightarrow{\mathcal F\text{-}o}0
\Longrightarrow
f(x_\alpha)\xrightarrow{o}0
\right\}.
\]
Similarly, the \emph{$\mathcal F\sigma$-order continuous dual} of
$X$ is defined by
\[
X^\sim_{\mathcal F\sigma}
=
\left\{
f\in X^\sim:
x_n\xrightarrow{\mathcal F\text{-}o}0
\Longrightarrow
f(x_n)\xrightarrow{o}0
\right\}.
\]

The following example shows how several familiar classes of
continuous operators arise from, or are related to,
$\mathcal{FG}$-order continuous operators.

\begin{example}\label{ex:special_cases_FG}
\begin{enumerate}

\item[(i)] For
\[
\mathcal F=\mathcal I_X
\qquad\text{and}\qquad
\mathcal G=\mathcal I_Y,
\]
$\mathcal{FG}$-order continuity coincides with order continuity.
Hence
\[
L_{\mathcal{FG}}(X,Y)=L_n(X,Y).
\]
Similarly,
\[
L_{\mathcal{FG}\sigma}(X,Y)=L_c(X,Y).
\]

\item[(ii)] For
\[
\mathcal F=\mathcal T_X
\qquad\text{and}\qquad
\mathcal G=\mathcal T_Y,
\]
$\mathcal{FG}$-order continuity coincides with $uo$-continuity.
Hence
\[
L_{\mathcal{FG}}(X,Y)=L_{uo}(X,Y).
\]

\item[(iii)] For
\[
\mathcal F=\mathcal T_X
\qquad\text{and}\qquad
\mathcal G=\mathcal I_Y,
\]
$\mathcal{FG}$-order continuity coincides with strongly order
continuity.

\item[(iv)] Let $X$ and $Y$ be normed Riesz spaces. For
\[
\mathcal F=\mathcal N_X
\qquad\text{and}\qquad
\mathcal G=\mathcal N_Y,
\]
$\mathcal{FG}$-order continuity coincides with norm continuity.

\item[(v)] Let $X$ and $Y$ be $f$-algebras. For
\[
\mathcal F=\mathcal M_X
\qquad\text{and}\qquad
\mathcal G=\mathcal M_Y,
\]
$\mathcal{FG}$-order continuous operators are precisely the
$mo$-continuous operators.

\item[(vi)] Let $X$ be a normed Riesz space and $Y$ a Riesz space.
For
\[
\mathcal F=\mathcal T_X\cup\mathcal N_X
\qquad\text{and}\qquad
\mathcal G=\mathcal T_Y,
\]
every boundedly unbounded order continuous (buo-continuous) operator
in the sense of \cite{BahramnezhadAzar18} is
$\mathcal{FG}$-order continuous.
Similarly, every $bu\sigma o$-continuous operator is
$\mathcal{FG}\sigma$-order continuous.

\item[(vii)] Suppose that $X^\sim$ and $Y^\sim$ separate the points
of $X$ and $Y$, respectively. For
\[
\mathcal F=\mathcal I_X
\qquad\text{and}\qquad
\mathcal G=\mathcal T_Y,
\]
every bbuo-continuous operator is $\mathcal{FG}$-order continuous.

\end{enumerate}
\end{example}

The preceding relationships are summarized in the following table.
\[
\boxed{
\begin{array}{ccl}
(\mathcal I_X,\mathcal I_Y)
&\Longleftrightarrow&
\text{order continuous},
\\[2mm]
(\mathcal T_X,\mathcal T_Y)
&\Longleftrightarrow&
uo\text{-continuous},
\\[2mm]
(\mathcal T_X,\mathcal I_Y)
&\Longleftrightarrow&
\text{strongly order continuous},
\\[2mm]
(\mathcal N_X,\mathcal N_Y)
&\Longleftrightarrow&
\text{norm continuous},
\\[2mm]
(\mathcal M_X,\mathcal M_Y)
&\Longleftrightarrow&
mo\text{-continuous}.
\end{array}
}
\]

For the bounded variants, we have the following implications:
\[
\boxed{
\begin{array}{ccl}
buo\text{-continuous}
&\Longrightarrow&
(\mathcal T_X\cup\mathcal N_X,\mathcal T_Y)
\text{-order continuous},
\\[2mm]
bbuo\text{-continuous}
&\Longrightarrow&
(\mathcal I_X,\mathcal T_Y)
\text{-order continuous}.
\end{array}
}
\]
In particular, since order convergence and $uo$-convergence
coincide in $\mathbb R$, the $uo$-continuous and strongly order
continuous duals coincide. Thus,
\[
X^\sim_{\mathcal T_X}
=
X^\sim_{uo}
=
X^\sim_{so}.
\]

The following example shows that the same operator may be $\mathcal{FG}$-order continuous for some choices of $\mathcal F$ and $\mathcal G$, but fail to be so for other choices.

\begin{example}\label{ex:Fourier_FG}
Consider the classical Fourier coefficient operator
\[
T:L_1[0,1]\longrightarrow c_0,\qquad
T(f)=
\left(
\int_0^1 f(x)\sin x\,dx,\,
\int_0^1 f(x)\sin 2x\,dx,\ldots
\right).
\]
It is known that $T$ is norm continuous but not order bounded
\cite[Exercise~10, Section~5.1]{AliprantisBurkinshaw06}.

Since $T$ is norm continuous, Example~\ref{ex:special_cases_FG}
implies that $T$ is $\mathcal{FG}$-order continuous for
\[
\mathcal F=\mathcal N_{L_1[0,1]}
\qquad\text{and}\qquad
\mathcal G=\mathcal N_{c_0}.
\]

On the other hand, since $T$ is not order bounded, it is not order
continuous. Hence, by Example~\ref{ex:special_cases_FG}, $T$ is not
$\mathcal{FG}$-order continuous for
\[
\mathcal F=\mathcal I_{L_1[0,1]}
\qquad\text{and}\qquad
\mathcal G=\mathcal I_{c_0}.
\]

Moreover, $T$ is not $uo$-continuous. Indeed, let
\[
I_k=
\left(
\frac12+2^{-k-1},
\frac12+2^{-k}
\right],
\qquad
f_k=2^{k+1}\chi_{I_k}.
\]
The functions $f_k$ are pairwise disjoint, and hence
\[
f_k\xrightarrow{uo}0
\qquad\text{in }L_1[0,1].
\]
However,
\[
(Tf_k)_1
=
2^{k+1}\int_{I_k}\sin x\,dx
\geq
\sin\frac12>0.
\]
Thus, $(Tf_k)$ does not converge coordinatewise to zero and therefore
does not $uo$-converge to zero in $c_0$. Hence $T$ is not
$uo$-continuous. By Example~\ref{ex:special_cases_FG}, $T$ is not
$\mathcal{FG}$-order continuous for
\[
\mathcal F=\mathcal T_{L_1[0,1]}
\qquad\text{and}\qquad
\mathcal G=\mathcal T_{c_0}.
\]

The same sequence shows that $T$ is not strongly order continuous,
since $Tf_k$ cannot order converge to zero. Consequently, $T$ is not
$\mathcal{FG}$-order continuous for
\[
\mathcal F=\mathcal T_{L_1[0,1]}
\qquad\text{and}\qquad
\mathcal G=\mathcal I_{c_0}.
\]

Finally, since $L_1[0,1]$ has order continuous norm, order convergence
in $L_1[0,1]$ implies norm convergence. Thus, if
\[
f_\alpha\xrightarrow{o}0,
\]
then
\[
\|f_\alpha\|_1\longrightarrow0.
\]
By the norm continuity of $T$,
\[
\|Tf_\alpha\|_\infty\longrightarrow0,
\]
and norm convergence in $c_0$ implies $uo$-convergence. Hence
\[
Tf_\alpha\xrightarrow{uo}0.
\]
Therefore, $T$ is $\mathcal{FG}$-order continuous for
\[
\mathcal F=\mathcal I_{L_1[0,1]}
\qquad\text{and}\qquad
\mathcal G=\mathcal T_{c_0}.
\]
\end{example}

Thus, the $\mathcal{FG}$-order continuity of an operator may depend
essentially on the choice of the families $\mathcal F$ and $\mathcal G$.
Moreover, the example also shows that $\mathcal{FG}$-order continuity
alone does not imply order boundedness. In particular, the
order-boundedness assumption in the definition of
$L_{\mathcal{FG}}(X,Y)$ is not redundant.

We shall use the following terminology. A family
$\mathcal F\subseteq\operatorname{IS}(X)$ is said to be
\emph{disjointly null} if every disjoint sequence in $X$ is
$\mathcal F$-order null. By \cite[Proposition~7]{Tantrawan25},
every family
\[
\mathcal F
\subseteq
\operatorname{IS}_b(X)\cap\operatorname{IS}_d(X)
\]
is disjointly null. More generally, the same conclusion holds
whenever every disjoint sequence in $X$ is order bounded and
$\mathcal F\subseteq\operatorname{IS}_d(X)$.

\begin{example}\label{ex:dual_operator}
Let $X=\ell_1$ and let $Y=s$ be the Riesz space of all real-valued
sequences equipped with the pointwise order. Consider the operator
\[
T:X\to Y,\qquad
T(x_n)=
\left(
x_1,x_1+x_2,\ldots,\sum_{i=1}^{n}x_i,\ldots
\right).
\]
Since $T$ is positive, its adjoint
$T^{\sim}:s^{\sim}\to\ell_1^{\sim}$ is order bounded and order
continuous \cite[Theorem~1.73]{AliprantisBurkinshaw06}.

We show that $T^{\sim}$ is not $\mathcal{FG}$-order continuous
for any disjointly null family
\[
\mathcal F\subseteq\operatorname{IS}(Y^\sim)
\]
and any
\[
\mathcal G\subseteq\operatorname{IS}_{\vee}(X^\sim)
\]
having the kernel separation property.

Indeed, for each $n\in\mathbb N$, define
\[
f_n:s\to\mathbb R,\qquad f_n(x)=x_n,
\]
where $x=(x_k)\in s$. Then $f_n\in s^\sim$, and $(f_n)$ is a
disjoint sequence. Since $\mathcal F$ is disjointly null,
\[
f_n\xrightarrow{\mathcal F\text{-}o}0.
\]

Let $e=(1,1,\ldots)$. Since
\[
T^\sim(f_n)\uparrow e
\quad\text{in }\ell_\infty\cong\ell_1^\sim,
\]
the monotonocity and the arbitrary-supremum preserving property of each
$G\in\mathcal G$ yield
\[
G\bigl(T^\sim(f_n)\bigr)\uparrow G(e).
\]

Thus,
\[
T^\sim(f_n)\xrightarrow{\mathcal G\text{-}o}e.
\]
Since $\mathcal G$ satisfies the kernel separation condition, the
$\mathcal G$-order limit is unique. As $e\neq0$,
$(T^\sim(f_n))$ is not $\mathcal G$-order null. Consequently,
$T^\sim$ is not $\mathcal{FG}$-order continuous.
\end{example}

In particular, taking
$\mathcal F=\mathcal T_{Y^\sim}$ and
$\mathcal G=\mathcal T_{X^\sim}$,
the $uo$-continuous case in
\cite[Example~3.6]{TuranAltinGurkok22}
is recovered.

We conclude this section with examples in which the families
$\mathcal F$ and $\mathcal G$ induce convergences different from
the classical ones.

\begin{example}\label{ex:direct_sum_FG}
Let $X=\ell_1\oplus\ell_1$ with the coordinatewise order, and let
\[
\mathcal F=\{F_z:z\in(\ell_1)_+\},
\qquad
F_z(x,y)=(x,y\wedge z).
\]
As observed in \cite[Example~2]{Tantrawan25}, the corresponding
$\mathcal F$-order convergence is, in general, neither order
convergence nor unbounded order convergence. Moreover,
\[
\bigcap_{F\in\mathcal F}\ker(F)=\{0\}.
\]
Taking $\mathcal G=\mathcal F$, the operator
\[
T:X\to X,\qquad T(x,y)=(0,y),
\]
is $\mathcal{FG}$-order continuous.
\end{example}

\begin{example}\label{ex:bounded_FG_families}
Let $X=\ell_1$, and let
\[
S(x_1,x_2,\ldots)=(0,x_1,x_2,\ldots).
\]
For $z\in X_+$, define
\[
F_z(x)=S(x)\wedge z,
\qquad
G_z(x)=S^2(x)\wedge z,
\]
and put
\[
\mathcal F=\{F_z:z\in X_+\},
\qquad
\mathcal G=\{G_z:z\in X_+\}.
\]
Then
\[
\bigcap_{F\in\mathcal F}\ker(F)
=
\bigcap_{G\in\mathcal G}\ker(G)
=
\{0\}.
\]
The shift operator $T:X\to X$, $T=S$, is
$\mathcal{FG}$-order continuous.
\end{example}

\section{Order Boundedness}

Abramovich and Sirotkin proved that every order continuous operator
between Riesz spaces is order bounded
\cite[Theorem~2.1]{AbramovichSirotkin05}. In contrast,
Example~\ref{ex:Fourier_FG} shows that
$\mathcal{FG}$-order continuity alone does not imply order boundedness.
We now give a sufficient condition under which order boundedness is
recovered.

We first recall a condition ensuring the linearity of
$\mathcal F$-order convergence. By
\cite[Proposition~5]{Tantrawan25}, $\mathcal F$-order convergence
is locally solid additive. Hence, by \cite[Theorem~2.1]{Bilokopytov23},
it is linear if and only if
\[
\frac{1}{n}x\xrightarrow{\mathcal F\text{-}o}0
\qquad\text{for every }x\in X_+.
\]
In particular, if
$\mathcal F\subseteq\operatorname{IS}_{\sigma o}(X)$, then
$\mathcal F$-order convergence is linear
\cite[Corollary~6]{Tantrawan25}.

The proof of the following theorem uses a blocking argument based on
a lexicographically ordered index set. Lexicographic combinations of
partially ordered sets were considered by Birkhoff \cite{Birkhoff37}.
In the context of order convergence, Wolk \cite{Wolk61} used such a
lexicographic construction as a blocking argument for nets. A similar
argument was later employed by Abramovich and Sirotkin
\cite[Theorem~2.1]{AbramovichSirotkin05}.

\begin{theorem}\label{thm:FG-continuous_order_bounded}
We use the lexicographic blocking argument described above. Let $X$ and $Y$ be Riesz spaces, and let
$\mathcal F\subseteq\operatorname{IS}(X)$. Suppose that
$\mathcal F$-order convergence on $X$ is linear. If an operator
$T:X\to Y$ satisfies
\[
x_\alpha\xrightarrow{\mathcal F\text{-}o}0
\quad\Longrightarrow\quad
Tx_\alpha\xrightarrow{o}0
\]
for every net $(x_\alpha)$ in $X$, then $T$ is order bounded.
\end{theorem}

\begin{proof}
Let $[0,b]\subseteq X$ be arbitrary. Set
\[
\Lambda=\mathbb N\times[0,b],
\]
equipped with the lexicographical order, that is,
\[
(n,y)\le (m,z)
\quad\text{if and only if}\quad
n<m
\quad\text{or}\quad
\bigl(n=m \text{ and } y\le z\bigr).
\]
Define
\[
x_{(n,y)}=\frac{1}{n}y,
\qquad (n,y)\in\Lambda.
\]
If $(m,y)\geq(n,b)$, then $m\geq n$, and hence
\[
0\leq x_{(m,y)}
=\frac{1}{m}y
\leq \frac{1}{n}b.
\]
Since $\mathcal F$-order convergence is linear,
$\frac{1}{n}b\xrightarrow{\mathcal F\text{-}o}0$. It follows from
\cite[Proposition~5(iv)]{Tantrawan25} that
\[
x_{(n,y)}\xrightarrow{\mathcal F\text{-}o}0.
\]

By hypothesis,
$Tx_{(n,y)}\xrightarrow{o}0$. Hence there exists a net
$(z_\beta)$ in $Y_+$ with $z_\beta\downarrow0$ such that, for every
$\beta$, there is $(n,y)\in\Lambda$ for which
\[
|Tx_{(m,u)}|\leq z_\beta
\qquad\text{whenever }(m,u)\geq(n,y).
\]
Fix $\beta$ and the corresponding $(n,y)$. For every $u\in[0,b]$,
we have $(n+1,u)\geq(n,y)$, and therefore
\[
\left|T\left(\frac{1}{n+1}u\right)\right|
\leq z_\beta.
\]
By linearity of $T$,
\[
|Tu|\leq(n+1)z_\beta
\qquad (u\in[0,b]).
\]
Thus $T([0,b])$ is order bounded. Since $[0,b]$ was arbitrary,
$T$ is order bounded.
\end{proof}

\begin{corollary}\label{cor:so_order_bounded}
Every strongly order continuous operator between Riesz spaces is
order bounded.
\end{corollary}

\begin{proof}
By Example~\ref{ex:special_cases_FG}, strongly order continuous
operators correspond to
\[
\mathcal F=\mathcal T_X,
\qquad
\mathcal G=\mathcal I_Y.
\]
Since $uo$-convergence, and hence $\mathcal T_X$-order convergence,
is linear, the result follows from
Theorem~\ref{thm:FG-continuous_order_bounded}.
\end{proof}

Therefore, the order
boundedness assumption in the definition of $L_{so}(X,Y)$ is
redundant, and we may simply write
\[
L_{so}(X,Y)
=
\{T:X\to Y : T \text{ is strongly order continuous}\}.
\]

\begin{corollary}\label{cor:mo_order_bounded}
Let $X$ be an $f$-algebra and let $Y$ be a unital $f$-algebra.
Then every $mo$-continuous operator $T:X\to Y$ is order bounded.
\end{corollary}

\begin{proof}
By Example~\ref{ex:special_cases_FG}, $mo$-continuity corresponds
to the multiplication families $\mathcal M_X$ and $\mathcal M_Y$.
By \cite{Aydin20}, $mo$-convergence is compatible with addition and
the lattice operations. Moreover, every quasi-subnet of an
$mo$-convergent net is $mo$-convergent to the same limit, since
order convergence has this property. Since $X$ is Archimedean,
$\frac{1}{n}x\xrightarrow{o}0$, and hence
$\frac{1}{n}x\xrightarrow{mo}0$ for every $x\in X_+$.
Thus, by \cite[Theorem~2.1]{Bilokopytov23},
$\mathcal M_X$-order convergence is linear.
Moreover, since $Y$ is unital, $\mathcal M_Y$-order convergence
coincides with order convergence. Hence the result follows from
Theorem~\ref{thm:FG-continuous_order_bounded}.
\end{proof}

\begin{remark}\label{rem:sequential_order_boundedness}
The sequential analogue of
Theorem~\ref{thm:FG-continuous_order_bounded} fails in general.
Indeed, for $\mathcal F=\mathcal I_X$, its sequential hypothesis
reduces to $\sigma$-order continuity. Such an analogue would therefore
imply that every $\sigma$-order continuous operator is order bounded,
which is false; see
\cite[Remark~2.2]{AbramovichSirotkin05}.
\end{remark}

\section{Lattice Properties}

To study the lattice properties of $L_{\mathcal{FG}}(X,Y)$, we establish
the following result, which plays a fundamental role in the study of
$\mathcal{FG}$-order continuity. Notably, no linearity or additivity
assumption is imposed on the mapping involved, making the result
applicable beyond the linear setting. The proof again uses the
lexicographic blocking argument described in Section~3.

\begin{theorem}\label{thm:supremum-envelope}
Let $X$ and $Y$ be Riesz spaces, with $Y$ Dedekind complete, and let
$Z$ be a Riesz subspace of $X$. Let
\[
\mathcal F\subseteq \operatorname{IS}(X)
\quad\text{and}\quad
\mathcal G\subseteq \operatorname{IS}_{\vee}(Y)
\]
satisfy
\[
\bigcap_{G\in\mathcal G}\ker(G)=\{0\}.
\]
Suppose that $\Phi:Z_+\to Y_+$ has the following property: for every
net $(z_\lambda)$ in $Z_+$,
\[
z_\lambda\xrightarrow{\mathcal F\text{-}o}0
\quad\text{in }X
\quad\Longrightarrow\quad
\Phi(z_\lambda)\xrightarrow{\mathcal G\text{-}o}0
\quad\text{in }Y.
\]
Assume also that, for every $x\in X_+$, the set
\[
\{\Phi(z):z\in Z_+,\ z\leq x\}
\]
is order bounded in $Y$. Define
\[
\widehat{\Phi}(x)
=
\sup\{\Phi(z):z\in Z_+,\ z\leq x\},
\qquad x\in X_+.
\]
Then, for every net $(x_\alpha)$ in $X_+$,
\[
x_\alpha\xrightarrow{\mathcal F\text{-}o}0
\quad\Longrightarrow\quad
\widehat{\Phi}(x_\alpha)
\xrightarrow{\mathcal G\text{-}o}0.
\]
\end{theorem}
\begin{proof}
Let $(x_\alpha)_{\alpha\in\Gamma}$ be a net in $X_+$ such that
$x_\alpha\xrightarrow{\mathcal F\text{-}o}0$. We consider two cases.

\medskip
\noindent
\textit{Case 1.} Suppose that $\Gamma$ has a maximal element
$\alpha_0$. Since $\Gamma$ is directed, $\alpha_0$ is the greatest
element of $\Gamma$. For every $F\in\mathcal F$,
$F(x_\alpha)\xrightarrow{o}0$, and hence $F(x_{\alpha_0})=0$.
Now let $z\in Z_+$ with $z\leq x_{\alpha_0}$. Since each
$F\in\mathcal F$ is increasing,
\[
0\leq F(z)\leq F(x_{\alpha_0})=0,
\]
so the constant net with value $z$ is $\mathcal F$-order null.
By the hypothesis on $\Phi$, the constant net with value $\Phi(z)$
is $\mathcal G$-order null. Since it also
$\mathcal G$-order converges to $\Phi(z)$ and
$\bigcap_{G\in\mathcal G}\ker G=\{0\}$, uniqueness of
$\mathcal G$-order limits yields $\Phi(z)=0$. Therefore
$\widehat{\Phi}(x_{\alpha_0})=0$. Since $\alpha_0$ is the greatest
element of $\Gamma$, it follows that
\[
\widehat{\Phi}(x_\alpha)
\xrightarrow{\mathcal G\text{-}o}0.
\]

\medskip
\noindent
\textit{Case 2.} Assume that $\Gamma$ has no maximal element. For each
$\alpha\in\Gamma$, set
\[
\Lambda_\alpha=\{(\alpha,u):u\in Z\cap[0,x_\alpha]\},
\qquad
\Lambda=\bigcup_{\alpha\in\Gamma}\Lambda_\alpha.
\]
We equip $\Lambda$ with the lexicographic order, that is,
$(\alpha,u)\leq(\alpha',u')$ if either $\alpha<\alpha'$, or
$\alpha=\alpha'$ and $u\leq u'$. Since $\Gamma$ is directed and has
no maximal element, $\Lambda$ is directed.

For $\lambda=(\alpha,u)\in\Lambda$, define
$z_\lambda=u$ and $\widetilde{x}_\lambda=x_\alpha$. Then
$0\leq z_\lambda\leq\widetilde{x}_\lambda$ for every
$\lambda\in\Lambda$. Moreover, for each $\alpha_0\in\Gamma$, if
$\lambda=(\alpha,u)\geq\lambda_0=(\alpha_0,0)$, then
$\alpha\geq\alpha_0$. Hence, by the definition of order convergence,
\[
x_\alpha\xrightarrow{\mathcal F\text{-}o}0
\quad\Longrightarrow\quad
\widetilde{x}_\lambda\xrightarrow{\mathcal F\text{-}o}0.
\]
Therefore, by the sandwich property of $\mathcal F$-order convergence
\cite[Proposition~5(iv)]{Tantrawan25},
$z_\lambda\xrightarrow{\mathcal F\text{-}o}0$.

By the hypothesis on $\Phi$,
$\Phi(z_\lambda)\xrightarrow{\mathcal G\text{-}o}0$. Fix
$G\in\mathcal G$. Then
$G(\Phi(z_\lambda))\xrightarrow{o}0$. Since $Y$ is Dedekind complete,
order convergence and strongly order convergence are equivalent.
Hence there exist $\lambda_1=(\alpha_1,u_1)\in\Lambda$ and a net
$(r_\lambda)$ in $Y_+$ such that $r_\lambda\downarrow0$ and
\[
G(\Phi(z_\lambda))\leq r_\lambda
\qquad(\lambda\geq\lambda_1).
\]
For each $\alpha\in\Gamma$, define
$s_\alpha=r_{(\alpha,0)}$. In fact, $s_\alpha\downarrow0$.
Indeed, $(s_\alpha)$ is decreasing. Moreover, since $\Gamma$ has no
maximal element, for every $\lambda=(\alpha,u)\in\Lambda$ there exists
$\alpha'>\alpha$, and hence
\[
s_{\alpha'}=r_{(\alpha',0)}\leq r_\lambda.
\]
Since $r_\lambda\downarrow0$, it follows that
$\inf_{\alpha\in\Gamma}s_\alpha=0$. Thus $s_\alpha\downarrow0$.

Choose $\alpha_2>\alpha_1$. If $\alpha\geq\alpha_2$ and
$u\in Z\cap[0,x_\alpha]$, then
$(\alpha,u)\geq(\alpha,0)\geq\lambda_1$, and therefore
\[
G(\Phi(u))
\leq r_{(\alpha,u)}
\leq r_{(\alpha,0)}
=s_\alpha.
\]
Since $G$ preserves arbitrary existing suprema,
\[
G(\widehat{\Phi}(x_\alpha))
=
\bigvee_{u\in Z\cap[0,x_\alpha]}G(\Phi(u))
\leq s_\alpha.
\]
Hence $G(\widehat{\Phi}(x_\alpha))\xrightarrow{o}0$. Since
$G\in\mathcal G$ was arbitrary,
\[
\widehat{\Phi}(x_\alpha)
\xrightarrow{\mathcal G\text{-}o}0.
\]
\end{proof}

\begin{corollary}\label{cor:modulus}
Let $X$ and $Y$ be Riesz spaces, with $Y$ Dedekind complete. Let
$\mathcal F\subseteq\operatorname{IS}(X)$ and
$\mathcal G\subseteq\operatorname{IS}_{\vee}(Y)$ satisfy
\[
\bigcap_{G\in\mathcal G}\ker(G)=\{0\}.
\]
If $T\in L_b(X,Y)$ is $\mathcal{FG}$-order continuous, then $|T|$ is
$\mathcal{FG}$-order continuous. Consequently,
$L_{\mathcal{FG}}(X,Y)$ is a Riesz subspace of $L_b(X,Y)$.
\end{corollary}

\begin{proof}

It is enough to show that $T^+$ is $\mathcal{FG}$-order continuous.

Define $\Phi:X_+\to Y_+$ by $\Phi(u)=|Tu|$. If
$u_\lambda\xrightarrow{\mathcal F\text{-}o}0$, then
$Tu_\lambda\xrightarrow{\mathcal G\text{-}o}0$, and hence
$|Tu_\lambda|\xrightarrow{\mathcal G\text{-}o}0$.

Since $T$ is order bounded, for every $x\in X_+$ the set
$\{|Tu|:0\leq u\leq x\}$ is order bounded in $Y$. Thus
Theorem~\ref{thm:supremum-envelope} yields
\[
\bigvee_{0\leq u\leq x_\alpha}|Tu|
\xrightarrow{\mathcal G\text{-}o}0
\]
whenever $0\leq x_\alpha\xrightarrow{\mathcal F\text{-}o}0$.
Since
\[
0\leq T^+(x_\alpha)
\leq\bigvee_{0\leq u\leq x_\alpha}|Tu|,
\]
we obtain
$T^+(x_\alpha)\xrightarrow{\mathcal G\text{-}o}0$. Hence $T^+$ is
$\mathcal{FG}$-order continuous.

Applying the same argument to $-T$, we obtain that $T^-$ is
$\mathcal{FG}$-order continuous. Since $|T|=T^++T^-$, the result
follows. Consequently, $L_{\mathcal{FG}}(X,Y)$ is a Riesz subspace
of $L_b(X,Y)$.
\end{proof}

The following consequences are immediate from 
Example~\ref{ex:special_cases_FG} and
Corollary~\ref{cor:modulus}.

\begin{corollary}\label{cor:riesz_subspace_special_cases}
Let $X$ be a Riesz space and let $Y$ be a Dedekind complete Riesz
space. Then the following statements hold.
\begin{enumerate}
\item[(i)] $L_{uo}(X,Y)$ is a Riesz subspace of $L_b(X,Y)$.

\item[(ii)] $L_{so}(X,Y)$ is a Riesz subspace of $L_b(X,Y)$.

\item[(iii)] Suppose that $X$ and $Y$ are $f$-algebras and $Y$ is
semiprime. Then $L_{mo}(X,Y)$ is a Riesz subspace of $L_b(X,Y)$.
\end{enumerate}
\end{corollary}

Part~(i) recovers the corresponding result for $uo$-continuous
operators obtained by Turan and G\"urk\"ok
\cite{TuranGurkok24}. Part~(ii) gives an affirmative answer to
Problem~1 of Bahramnezhad and Haghnejad Azar
\cite{Bahramnezhad_Azar}. G\"urk\"ok and Turan
\cite[Corollaries~3.7 and~3.8]{GurkokTuran25} obtained related
modulus and Riesz space results for $mo$-continuous operators under
different assumptions.

\begin{remark}\label{rem:Fsigma_Riesz_space}
For the sequential case, some special cases are already known. In
particular, Bahramnezhad and Haghnejad Azar
\cite[Theorem~1 and Remark~1]{Bahramnezhad_Azar} showed that the
modulus of an order bounded strongly $\sigma$-order continuous
functional is again strongly $\sigma$-order continuous. The same
argument shows that $X^\sim_{\mathcal F\sigma}$ is a Riesz subspace
of $X^\sim$.

Indeed, let $f\in X^\sim_{\mathcal F\sigma}$ and
$0\leq x_n\xrightarrow{\mathcal F\text{-}o}0$. Choose
$u_n\in[0,x_n]$ such that
\[
f^+(x_n)<f(u_n)+\frac1n.
\]
Then $u_n\xrightarrow{\mathcal F\text{-}o}0$, and hence
\[
0\leq f^+(x_n)<f(u_n)+\frac1n\longrightarrow0.
\]
Thus $f^+\in X^\sim_{\mathcal F\sigma}$, and consequently
$X^\sim_{\mathcal F\sigma}$ is a Riesz subspace of $X^\sim$.

The general case, however, remains open and leads to the following
problem.
\end{remark}

\begin{problem}\label{prob:modulus_FGsigma}
Let $X$ and $Y$ be Riesz spaces, with $Y$ Dedekind complete, and let
\[
\mathcal F\subseteq\operatorname{IS}(X)
\quad\text{and}\quad
\mathcal G\subseteq\operatorname{IS}(Y).
\]
If $T\in L_b(X,Y)$ is $\mathcal{FG}\sigma$-order continuous, is
$|T|$ necessarily $\mathcal{FG}\sigma$-order continuous?
\end{problem}

\section{Order Continuity}

In this section, we study the relationship between
$\mathcal{FG}$-order continuity and classical order continuity.
We first establish a useful property of $\mathcal F$-order
convergence that will be needed in our study of
$\mathcal{FG}$-order continuous operators.

\begin{lemma}\label{lem:order-bounded-F-null}
Let $X$ be an Archimedean Riesz space and let
\[
\mathcal F\subseteq\operatorname{IS}_{\vee}(X)
\]
satisfy
\[
\bigcap_{F\in\mathcal F}\ker F=\{0\}.
\]
Then every order bounded $\mathcal F$-order null net in $X$
is order null.
\end{lemma}

\begin{proof}
First suppose that $X$ is Dedekind complete. Let $(x_\alpha)$ be
order bounded and
$x_\alpha\xrightarrow{\mathcal F\text{-}o}0$. For each $\beta$, set
\[
u_\beta=\sup_{\alpha\geq\beta}|x_\alpha|.
\]
Then $(u_\beta)$ is decreasing. Since
$F(|x_\alpha|)\xrightarrow{o}0$ for every $F\in\mathcal F$ and
$F$ preserves arbitrary existing suprema, we have
\[
F(u_\beta)
=
\sup_{\alpha\geq\beta}F(|x_\alpha|)
\downarrow0.
\]
Let $u=\inf_\beta u_\beta$. Then
$0\leq F(u)\leq F(u_\beta)$ for every $\beta$, and hence $F(u)=0$
for every $F\in\mathcal F$. By the kernel separation property,
$u=0$. Thus $u_\beta\downarrow0$, and therefore
$x_\alpha\xrightarrow{o}0$.

Now let $X$ be arbitrary, and let $X^\delta$ be its Dedekind
completion. For each $F\in\mathcal F$, let
\[
F^\delta(u)
=
\sup\{F(x):x\in X_+,\ x\leq u\},
\qquad u\in(X^\delta)_+,
\]
be the extension introduced by Tantrawan \cite{Tantrawan25}.
The kernel separation property of $\mathcal F$ clearly passes to
the family
\[
\mathcal F^\delta=\{F^\delta:F\in\mathcal F\}.
\]

We claim that
$F^\delta\in\operatorname{IS}_{\vee}(X^\delta)$ for every
$F\in\mathcal F$. Let $A\subseteq(X^\delta)_+$ and suppose that
$u=\sup A$. Set
\[
B=\{y\in X_+:\text{ there exists }a\in A\text{ with }y\leq a\}.
\]
Then
\[
\sup B
=
\sup_{a\in A}\sup\{y\in X_+:y\leq a\}
=
\sup A
=
u.
\]

Let $0\leq x\in X$ with $x\leq u$. As $x\leq u$ and $u=\sup B$, the infinite distributive law gives
\[
x
=x\wedge u
=x\wedge\sup_{y\in B}y
=\sup_{y\in B}(x\wedge y).
\]

Since $X$ is regular in $X^\delta$,
this supremum is also taken in $X$. Hence, using
$F\in\operatorname{IS}_{\vee}(X)$,
\[
F(x)
=
\sup_{y\in B}F(x\wedge y)
\leq
\sup_{a\in A}F^\delta(a).
\]
Taking the supremum over all $x\in X_+$ with $x\leq u$, we obtain
\[
F^\delta(u)\leq\sup_{a\in A}F^\delta(a).
\]
The reverse inequality follows from monotonicity. Thus
$F^\delta\in\operatorname{IS}_{\vee}(X^\delta)$.
By \cite[Corollary~17]{Tantrawan25},
\[
x_\alpha\xrightarrow{\mathcal F\text{-}o}0
\quad\Longrightarrow\quad
x_\alpha\xrightarrow{\mathcal F^\delta\text{-}o}0
\quad\text{in }X^\delta.
\]
Since $(x_\alpha)$ is order bounded in $X$, it is also order bounded
in $X^\delta$. Therefore, applying the first part of the proof to
$X^\delta$ and $\mathcal F^\delta$, we obtain
\[
x_\alpha\xrightarrow{o}0
\quad\text{in }X^\delta.
\]
By \cite[Corollary~2.9]{GaoTroitskyXanthos17}, this is equivalent to
\[
x_\alpha\xrightarrow{o}0
\quad\text{in }X.
\]
This completes the proof.
\end{proof}

\begin{theorem}\label{thm:FG-continuous-implies-order-continuous}
Let $T:X\to Y$ be an order bounded operator between
Riesz spaces. Let
\[
\mathcal F\subseteq\operatorname{IS}_o(X)
\quad\text{and}\quad
\mathcal G\subseteq\operatorname{IS}_{\vee}(Y)
\]
satisfy
\[
\bigcap_{G\in\mathcal G}\ker G=\{0\}.
\]
If $T$ is $\mathcal{FG}$-order continuous, then $T$ is order
continuous.
\end{theorem}

\begin{proof}
Let $x_\alpha\xrightarrow{o}0$ in $X$. Since
$\mathcal F\subseteq\operatorname{IS}_o(X)$, we have
$x_\alpha\xrightarrow{\mathcal F\text{-}o}0$. Since $T$ is
$\mathcal{FG}$-order continuous,
\[
Tx_\alpha\xrightarrow{\mathcal G\text{-}o}0.
\]

The net $(x_\alpha)$ is eventually order bounded, and hence so is
$(Tx_\alpha)$, since $T$ is order bounded. Applying
Lemma~\ref{lem:order-bounded-F-null} to an order bounded tail of
$(Tx_\alpha)$, we obtain
\[
Tx_\alpha\xrightarrow{o}0.
\]
Thus $T$ is order continuous.
\end{proof}

Combining Example~\ref{ex:special_cases_FG} with
Theorem~\ref{thm:FG-continuous-implies-order-continuous}, we obtain
the following consequences.

\begin{corollary}\label{cor:consequences-order-continuity}
The following statements hold.
\begin{enumerate}

\item[(i)] Every order bounded $uo$-continuous operator
$T:X\to Y$ between Riesz spaces is order continuous.

\item[(ii)] Every strongly order continuous operator
$T:X\to Y$ between Riesz spaces is order continuous. In particular,
\[
L_{so}(X,Y)\subseteq L_n(X,Y).
\]

\item[(iii)] Let $X$ and $Y$ be normed Riesz spaces. If the norm of
$X$ is order continuous and the norm of $Y$ preserves arbitrary
suprema, then every order bounded norm continuous operator
$T:X\to Y$ is order continuous.

\item[(iv)] Let $X$ and $Y$ be $f$-algebras, with $Y$ semiprime.
Then every order bounded $mo$-continuous operator
$T:X\to Y$ is order continuous.

\item[(v)] Let $X$ be a normed Riesz space with order continuous norm
and let $Y$ be a Riesz space. Then every order bounded
$buo$-continuous operator $T:X\to Y$ is order continuous.

\item[(vi)] Let $X$ and $Y$ be Riesz spaces whose order duals
separate points. Then every order bounded $bbuo$-continuous operator
$T:X\to Y$ is order continuous.

\end{enumerate}
\end{corollary}

Part~\textup{(i)} extends
\cite[Proposition~3.2]{TuranAltinGurkok22} by removing the
Dedekind completeness assumption on $Y$.
When $Y$ is Dedekind complete, Corollary~\ref{cor:modulus} also
gives the corresponding modulus result for order bounded
$uo$-continuous operators.

Part~\textup{(iv)} recovers
\cite[Proposition~3.3]{GurkokTuran25}.
To the best of our knowledge, the conclusions in
parts~\textup{(ii)}, \textup{(v)}, and~\textup{(vi)} have not
previously been established for strongly order continuous,
$buo$-continuous, and $bbuo$-continuous operators, respectively.

\begin{remark}
The converse of
Theorem~\ref{thm:FG-continuous-implies-order-continuous}
does not hold in general, even for lattice homomorphisms.
Indeed, let $T$ be the identity operator on $c_0$. Then
\[
e_n\xrightarrow{uo}0,
\qquad\text{but}\qquad
e_n\not\xrightarrow{o}0.
\]
Hence, $T$ is order continuous but not strongly order continuous.
Therefore, by Example~\ref{ex:special_cases_FG}, $T$ is not
$\mathcal{FG}$-order continuous for
\[
\mathcal F=\mathcal T_{c_0}
\qquad\text{and}\qquad
\mathcal G=\mathcal I_{c_0}.
\]
\end{remark}

\begin{theorem}
\label{thm:FGsigma-continuous-implies-sigma-order-continuous}
Let $T:X\to Y$ be an order bounded operator between Riesz spaces.
Let
\[
\mathcal F\subseteq\operatorname{IS}_{\sigma o}(X)
\quad\text{and}\quad
\mathcal G\subseteq\operatorname{IS}_{\vee}(Y)
\]
satisfy
\[
\bigcap_{G\in\mathcal G}\ker G=\{0\}.
\]
If $T$ is $\mathcal{FG}\sigma$-order continuous, then $T$ is
$\sigma$-order continuous.
\end{theorem}

\begin{proof}
The proof follows from the same argument as in
Theorem~\ref{thm:FG-continuous-implies-order-continuous}, applied
to sequences.
\end{proof}

\section{Band Properties and Finite-Dimensionality}

Having shown that $L_{\mathcal{FG}}(X,Y)$ is a Riesz subspace of
$L_b(X,Y)$, we now study its ideal and band properties. We first
show that, under suitable assumptions, it is an ideal but need not
be a band, and then relate the band property to the
finite-dimensionality of $X$.

\begin{proposition}\label{prop:ideal_properties}
Let $X$ and $Y$ be Riesz spaces, with $Y$ Dedekind complete, and let
$\mathcal F\subseteq\operatorname{IS}(X)$ and
$\mathcal G\subseteq\operatorname{IS}_{\vee}(Y)$ satisfy
$\bigcap_{G\in\mathcal G}\ker G=\{0\}$.
Let $T,S\in L_b(X,Y)$ with $|S|\leq|T|$.
If $T$ is $\mathcal{FG}$-order continuous, then $S$ is also
$\mathcal{FG}$-order continuous.
\end{proposition}

\begin{proof}
Let $0\leq x_\alpha\xrightarrow{\mathcal F\text{-}o}0$.
By Corollary~\ref{cor:modulus},
$|T|x_\alpha\xrightarrow{\mathcal G\text{-}o}0$.
For every $\alpha$,
\[
|Sx_\alpha|\leq |S|x_\alpha\leq |T|x_\alpha.
\]
Hence, for every $G\in\mathcal G$,
\[
0\leq G(|Sx_\alpha|)
\leq G(|T|x_\alpha)\xrightarrow{o}0.
\]
Thus $Sx_\alpha\xrightarrow{\mathcal G\text{-}o}0$, and therefore
$S$ is $\mathcal{FG}$-order continuous.
\end{proof}

The following consequences are immediate from 
Example~\ref{ex:special_cases_FG} and
Proposition~\ref{prop:ideal_properties}.

\begin{corollary}\label{cor:ideal_special_cases}
Let $X$ and $Y$ be Riesz spaces, with $Y$ Dedekind complete. Then
the following statements hold.
\begin{enumerate}
\item[(i)] $L_{uo}(X,Y)$ is an ideal of $L_b(X,Y)$.

\item[(ii)] $L_{so}(X,Y)$ is an ideal of $L_b(X,Y)$.

\item[(iii)] Suppose that $X$ and $Y$ are $f$-algebras and $Y$ is
semiprime. Then $L_{mo}(X,Y)$ is an ideal of $L_b(X,Y)$.
\end{enumerate}
\end{corollary}

Part~(i) improves \cite[Corollary~3]{TuranGurkok24}.
For strongly order continuous operators, the band property has
previously been studied in the scalar-valued case $Y=\mathbb R$;
see \cite{Bahramnezhad_Azar}.

In general, $L_{\mathcal{FG}}(X,Y)$ need not be a band in
$L_b(X,Y)$, even though it is an ideal.
In particular, the space of order bounded $uo$-continuous operators
need not be a band; see
\cite[Example~2.12]{TuranAltinGurkok22}. The following example shows
that $L_{so}(X,Y)$ need not be a band in $L_b(X,Y)$ either.

\begin{example}\label{ex:linfty_not_band}
Let $X=Y=\ell_\infty$, and take
$\mathcal F=\mathcal T_{\ell_\infty}$ and
$\mathcal G=\mathcal I_{\ell_\infty}$.
For each $n\in\mathbb N$, define $P_n:\ell_\infty\to\ell_\infty$ by
\[
P_n(x_1,x_2,\ldots)
=
(x_1,\ldots,x_n,0,0,\ldots).
\]
Then $P_n$ is a band projection, and hence an order continuous lattice
homomorphism. Therefore, by \cite[Theorem~3.3]{TuranAltinGurkok22},
$P_n$ is $uo$-continuous. Thus, if
$x_\alpha\xrightarrow{uo}0$ in $\ell_\infty$, then
\[
P_nx_\alpha\xrightarrow{uo}0.
\]
Since $uo$-convergence in $\ell_\infty$ is coordinatewise convergence,
we have
\[
x_\alpha(i)\longrightarrow0
\qquad (i=1,\ldots,n).
\]
Hence, for each $m\in\mathbb N$, there exists $\alpha_m$ such that
\[
|x_\alpha(i)|\leq\frac1m
\qquad
(\alpha\geq\alpha_m,\; i=1,\ldots,n).
\]
Therefore,
\[
|P_nx_\alpha|
\leq\frac1m\sum_{i=1}^n e_i
\qquad(\alpha\geq\alpha_m).
\]
Since
\[
\frac1m\sum_{i=1}^n e_i\downarrow0,
\]
it follows that
\[
P_nx_\alpha\xrightarrow{o}0.
\]
Therefore,
$P_n\in L_{\mathcal{FG}}(\ell_\infty)$ for every $n\in\mathbb N$.

On the other hand, $0\leq P_n\uparrow I$ in
$L_b(\ell_\infty)$. However, $I$ is not
$\mathcal{FG}$-order continuous. Indeed, the sequence
$x_n=ne_n$ is disjoint and hence $x_n\xrightarrow{uo}0$.
Since $(x_n)$ is not eventually order bounded, it cannot order
converge to zero. Thus
\[
Ix_n\not\xrightarrow{o}0,
\]
and consequently $I\notin L_{\mathcal{FG}}(\ell_\infty)$.
Therefore, $L_{\mathcal{FG}}(\ell_\infty)$ is not a band in
$L_b(\ell_\infty)$.
\end{example}

We next record a finite-dimensional observation. In
\cite[Exercise~84.6]{Zaanen83}, it is shown that every linear operator
from a finite-dimensional Archimedean Riesz space into a Dedekind
complete Riesz space is order continuous. The following lemma shows
that the Dedekind completeness assumption on the range space is not
needed.

\begin{lemma}\label{lem:finite_dimensional_order_continuity}
Let $X$ be a finite-dimensional Archimedean Riesz space and let $Y$ be
a Riesz space. Then every linear operator $T:X\to Y$ is order
continuous. In particular,
\[
L_b(X,Y)=L_n(X,Y)=L_c(X,Y).
\]
\end{lemma}

\begin{proof}
Since $X$ is a finite-dimensional Archimedean Riesz space, we may
identify $X$ with $\mathbb R^n$ equipped with its coordinatewise
order. Let $x_\alpha\xrightarrow{o}0$ in $X$. Then
\[
x_\alpha(i)\longrightarrow0
\qquad (i=1,\ldots,n).
\]
Hence, for each $m\in\mathbb N$, there exists $\alpha_m$ such that
\[
|x_\alpha(i)|\leq\frac1m
\qquad
(\alpha\geq\alpha_m,\; i=1,\ldots,n).
\]
Let $(e_i)_{i=1}^n$ denote the standard unit vectors and set
\[
u=\sum_{i=1}^n|Te_i|.
\]
Then, for $\alpha\geq\alpha_m$,
\[
|Tx_\alpha|
\leq\sum_{i=1}^n |x_\alpha(i)|\,|Te_i|
\leq\frac1m u.
\]
Since $Y$ is Archimedean, $\frac1m u\downarrow0$. Therefore
$Tx_\alpha\xrightarrow{o}0$, and hence $T$ is order continuous.

Thus every order bounded operator from $X$ into $Y$ is order
continuous, and consequently
\[
L_b(X,Y)=L_n(X,Y).
\]
Moreover, every finite-dimensional Riesz space is order separable.
Hence, by \cite[Theorem~84.4]{Zaanen83},
\[
L_c(X,Y)\subseteq L_n(X,Y).
\]
Since $L_n(X,Y)\subseteq L_c(X,Y)$, we conclude that
\[
L_b(X,Y)=L_n(X,Y)=L_c(X,Y).
\]
\end{proof}
\begin{lemma}\label{lem:finite_dimensional_F_order}
Let $X$ be a finite-dimensional Riesz space, and let
$\mathcal F\subseteq\operatorname{IS}(X)$ satisfy the kernel
separation condition.
Then
\[
x_\alpha\xrightarrow{\mathcal F\text{-}o}0
\quad\Longrightarrow\quad
x_\alpha\xrightarrow{o}0
\]
for every net $(x_\alpha)$ in $X$.
If, in addition, $\mathcal F\subseteq\operatorname{IS}_o(X)$,
then $\mathcal F$-order convergence and order convergence coincide
on $X$.
\end{lemma}

\begin{proof}
By \cite[Theorem~26.11]{LuxemburgZaanen71}, every
finite-dimensional Archimedean Riesz space is Riesz isomorphic to
$\mathbb R^n$ with the coordinatewise order. Thus, it suffices to
consider $X=\mathbb R^n$.

Since order convergence in $\mathbb R^n$ is coordinatewise convergence,
there exists $i\in\{1,\ldots,n\}$ such that
$x_\alpha(i)\not\to0$. Hence there exists $\varepsilon>0$ such that,
for every $\alpha_0$, there is some $\alpha\geq\alpha_0$ satisfying
\[
|x_\alpha(i)|\geq\varepsilon,
\]
and therefore
\[
\varepsilon e_i\leq |x_\alpha|.
\] By the kernel separation property, there exists
$F_0\in\mathcal F$ such that
$F_0(\varepsilon e_i)>0$. Since $F_0$ is increasing, for every
$\alpha_0$ there exists $\alpha\geq\alpha_0$ such that
\[
0<F_0(\varepsilon e_i)\leq F_0(|x_\alpha|).
\]
This contradicts
$F_0(|x_\alpha|)\xrightarrow{o}0$. Hence
$x_\alpha\xrightarrow{o}0$.

If $\mathcal F\subseteq\operatorname{IS}_o(X)$, the converse follows
directly from the order continuity of every $F\in\mathcal F$.
\end{proof}

\begin{theorem}\label{thm:finite_dimensional_FG}
Let $X$ and $Y$ be Riesz spaces, let
$\mathcal F\subseteq\operatorname{IS}(X)$ satisfy the kernel
separation condition, and let
$\mathcal G\subseteq\operatorname{IS}_o(Y)$.
If $X$ is finite-dimensional, then
\[
L_{\mathcal{FG}}(X,Y)=L_n(X,Y).
\]
\end{theorem}

\begin{proof}
Let $T\in L_n(X,Y)$ and suppose that
$x_\alpha\xrightarrow{\mathcal F\text{-}o}0$.
By Lemma~\ref{lem:finite_dimensional_F_order},
\[
x_\alpha\xrightarrow{o}0.
\]
Since $T$ is order continuous,
$Tx_\alpha\xrightarrow{o}0$. As
$\mathcal G\subseteq\operatorname{IS}_o(Y)$, it follows that
\[
Tx_\alpha\xrightarrow{\mathcal G\text{-}o}0.
\]
Hence
\[
L_n(X,Y)\subseteq L_{\mathcal{FG}}(X,Y).
\]

Conversely, by definition,
$L_{\mathcal{FG}}(X,Y)\subseteq L_b(X,Y)$. Since $X$ is
finite-dimensional, Lemma~\ref{lem:finite_dimensional_order_continuity}
gives
\[
L_b(X,Y)=L_n(X,Y).
\]
Therefore
\[
L_{\mathcal{FG}}(X,Y)=L_n(X,Y).
\]
\end{proof}

Recall that a family
$\mathcal F\subseteq\operatorname{IS}(X)$ is \emph{disjointly null}
if every disjoint sequence in $X$ is $\mathcal F$-order null.
We next give a sufficient condition under which the band property of
$L_{\mathcal{FG}}(X,Y)$ forces $X$ to be finite-dimensional.

\begin{theorem}\label{thm:FG_band_finite_dimensional}
Let $X$ be a Banach lattice and let $Y$ be a Riesz space. Let
$\mathcal F\subseteq\operatorname{IS}(X)$ be disjointly null, and let
$\mathcal G\subseteq\operatorname{IS}_o(Y)$ satisfy the kernel
separation condition. Suppose that $X^\sim_{\mathcal F}$ separates
the points of $X$. If $L_{\mathcal{FG}}(X,Y)$ is a band in
$L_b(X,Y)$, then $X$ is finite-dimensional.
\end{theorem}

\begin{proof}
Suppose that $X$ is infinite-dimensional. Since $X$ is Archimedean,
\cite[Theorem~26.10]{LuxemburgZaanen71} shows that $X$ contains an
infinite disjoint system of nonzero elements. Thus, we may choose a
disjoint sequence
\[
(x_n)\subseteq X_+\setminus\{0\}.
\]
Since $X^\sim_{\mathcal F}$ separates the points of $X$, for each
$n\in\mathbb N$ there exists
\[
f_n\in (X^\sim_{\mathcal F})_+
\]
such that $f_n(x_n)=1$. Each $f_n$ is norm continuous, since every
positive linear functional on a Banach lattice is continuous. 

Fix $0<u\in Y$ and define
\[
T_n(x)=
\left(
\sum_{i=1}^{n}\frac{f_i(x)}{2^i\|f_i\|}
\right)u,
\qquad x\in X.
\]
Since each $f_i\in X^\sim_{\mathcal F}$ and
$\mathcal G\subseteq\operatorname{IS}_o(Y)$,
\[
T_n\in L_{\mathcal{FG}}(X,Y)
\qquad(n\in\mathbb N).
\]

The series
\[
f=\sum_{i=1}^{\infty}\frac{f_i}{2^i\|f_i\|}
\]
converges in norm. Define $T(x)=f(x)u$ for $x\in X$. Then
$T\in L_b(X,Y)$ and
\[
0\leq T_n\uparrow T.
\]
Since $L_{\mathcal{FG}}(X,Y)$ is a band in $L_b(X,Y)$, we obtain
$T\in L_{\mathcal{FG}}(X,Y)$.

Now set
\[
y_n=2^n\|f_n\|x_n.
\]
The sequence $(y_n)$ is disjoint. Since $\mathcal F$ is disjointly
null,
\[
y_n\xrightarrow{\mathcal F\text{-}o}0.
\]
On the other hand,
\[
T(y_n)
\geq
\frac{f_n(y_n)}{2^n\|f_n\|}u
=u.
\]
Since $\mathcal G$ satisfies the kernel separation condition and
$u>0$, there exists $G_0\in\mathcal G$ such that $G_0(u)>0$.
By the monotonicity of $G_0$,
\[
G_0(T(y_n))\geq G_0(u)>0
\qquad(n\in\mathbb N).
\]
Thus
\[
T(y_n)\not\xrightarrow{\mathcal G\text{-}o}0,
\]
contradicting $T\in L_{\mathcal{FG}}(X,Y)$. Hence $X$ is
finite-dimensional.
\end{proof}

\begin{corollary}\label{cor:FG_band_characterization}
Let $X$ be a Banach lattice and let $Y$ be a Dedekind
complete Riesz space. Let
$\mathcal F\subseteq\operatorname{IS}(X)$ be disjointly null and
satisfy the kernel separation condition, and let
$\mathcal G\subseteq\operatorname{IS}_o(Y)$ satisfy the kernel
separation condition. Suppose that $X^\sim_{\mathcal F}$ separates
the points of $X$. Then the following statements are equivalent:
\begin{enumerate}
\item[(i)] $X$ is finite-dimensional;
\item[(ii)] $L_{\mathcal{FG}}(X,Y)=L_n(X,Y)$;
\item[(iii)] $L_{\mathcal{FG}}(X,Y)$ is a band in $L_b(X,Y)$.
\end{enumerate}
\end{corollary}

\begin{proof}
The implication $(i)\Rightarrow(ii)$ follows from
Theorem~\ref{thm:finite_dimensional_FG}. If $(ii)$ holds, then
\cite[Theorem~1.57]{AliprantisBurkinshaw06} implies that
$L_{\mathcal{FG}}(X,Y)=L_n(X,Y)$ is a band in $L_b(X,Y)$.
Thus $(ii)\Rightarrow(iii)$. Finally,
$(iii)\Rightarrow(i)$ follows from
Theorem~\ref{thm:FG_band_finite_dimensional}.
\end{proof}
\begin{remark}\label{rem:band_conditions}
\begin{enumerate}
\item[(i)] The disjointly null assumption in
Theorem~\ref{thm:FG_band_finite_dimensional} cannot, in general, be
omitted. Indeed, let $X=\ell_\infty$, $\mathcal F=\mathcal I_X$,
$Y=\mathbb R$, and $\mathcal G=\mathcal I_{\mathbb R}$. Then
\[
L_{\mathcal{FG}}(X,Y)=X_n^\sim
\]
is a band in $X^\sim=L_b(X,\mathbb R)$, although $X$ is
infinite-dimensional. On the other hand, $\mathcal F$ is not
disjointly null, since the disjoint sequence $(ne_n)$ is not
order null.

\item[(ii)] Recall that $X^\sim_{\mathcal F\sigma}$ is a Riesz space by
Remark~\ref{rem:Fsigma_Riesz_space}. With the appropriate modifications,
Theorems~\ref{thm:finite_dimensional_FG}
and~\ref{thm:FG_band_finite_dimensional}, as well as
Corollary~\ref{cor:FG_band_characterization}, also hold for
$\mathcal{FG}\sigma$-order continuous operators.
\end{enumerate}
\end{remark}

We now have the following results.

\begin{corollary}\label{cor:uo_so_band_characterizations}
Let $X$ be a Banach lattice and let $Y$ be a Dedekind complete
Riesz space. Suppose that $X^\sim_{uo}$ separates the points of $X$.
Then the following statements are equivalent:

\begin{enumerate}
\item[(i)] $X$ is finite-dimensional;
\item[(ii)] $L_{uo}(X,Y)=L_n(X,Y)$;
\item[(iii)] $L_{so}(X,Y)=L_n(X,Y)$;
\item[(iv)] $L_{uo}(X,Y)$ is a band in $L_b(X,Y)$;
\item[(v)] $L_{so}(X,Y)$ is a band in $L_b(X,Y)$.
\end{enumerate}
\end{corollary}

The equivalence of \textup{(i)}, \textup{(ii)}, and \textup{(iii)}
recovers \cite[Proposition~5]{TuranGurkok24}.

Notice that $\mathcal M_X$ need not be disjointly null; for instance,
the disjoint sequence $(ne_n)$ in $\ell_\infty$ is not $mo$-null
\cite[Example~2.8]{GurkokTuran25}. Nevertheless, we obtain the
following finite-dimensional result.

\begin{corollary}\label{cor:mo_band_finite_dimensional}
Let $X$ be a finite-dimensional semiprime $f$-algebra and let
$Y$ be an $f$-algebra. Then
\[
L_{mo}(X,Y)=L_n(X,Y).
\]
If, in addition, $Y$ is Dedekind complete, then
$L_{mo}(X,Y)$ is a band in $L_b(X,Y)$.
\end{corollary}

\begin{corollary}\label{cor:F_order_finite_dimensional}
Let $X$ be a Banach lattice and let
$\mathcal F\subseteq\operatorname{IS}(X)$ be disjointly null.
Suppose that $X^\sim_{\mathcal F}$ separates the points of $X$.
If $\mathcal F$-order convergence and order convergence coincide
in $X$, then $X$ is finite-dimensional.
\end{corollary}

\begin{proof}
Suppose that $\mathcal F$-order convergence and
order convergence coincide in $X$. Let
\[
Y=\mathbb R
\quad\text{and}\quad
\mathcal G=\mathcal I_{\mathbb R}.
\]
Then
\[
X^\sim_{\mathcal F}=X_n^\sim.
\]
Since $X_n^\sim$ is a band in $X^\sim$,
Theorem~\ref{thm:FG_band_finite_dimensional} implies that $X$ is
finite-dimensional.
\end{proof}

\section{Extension and Ideal Results}

We conclude the paper by studying extension properties of
$\mathcal{FG}$-order continuous operators. Using
Theorem~\ref{thm:supremum-envelope}, we first obtain an
$\mathcal{FG}$-order continuous version of Veksler's extension
theorem and then characterize when a positive operator is
$\mathcal{FG}$-order continuous on a given ideal.

\begin{theorem}\label{thm:extension-theorem}
Let $Z$ be an order dense majorizing Riesz subspace of a Riesz
space $X$, and let $Y$ be a Dedekind complete Riesz space. Let
$\mathcal F\subseteq\operatorname{IS}_o(X)$ satisfy
\[
F(Z_+)\subseteq Z_+
\qquad\text{for every }F\in\mathcal F.
\]
Let
\[
\mathcal G\subseteq\operatorname{IS}_{\vee}(Y)
\]
have the kernel separation property.
If $T:Z\to Y$ is a positive
$\mathcal{FG}$-order continuous operator, then $T$
possesses a unique positive $\mathcal{FG}$-order continuous
extension to all of $X$.
\end{theorem}

\begin{proof}
Since $T$ is positive, it is order bounded. Moreover, the restrictions
of the members of $\mathcal F$ to $Z_+$ belong to
$\operatorname{IS}_o(Z)$. Hence
Theorem~\ref{thm:FG-continuous-implies-order-continuous} implies that
$T$ is order continuous.

By Veksler's extension theorem
\cite[Theorem~1.65]{AliprantisBurkinshaw06}, the formula
\[
\widehat T(x)
=
\sup\{T(u):u\in Z,\ 0\leq u\leq x\},
\qquad x\in X_+,
\]
defines the unique positive order continuous extension
$\widehat T:X\to Y$ of $T$.

Define $\Phi:Z_+\to Y_+$ by $\Phi(u)=T(u)$. Let
$(z_\lambda)\subseteq Z_+$ satisfy
\[
z_\lambda\xrightarrow{\mathcal F\text{-}o}0
\quad\text{in }X.
\]
Since $Z$ is order dense and majorizing,
\cite[Theorem~15]{Tantrawan25} yields
\[
z_\lambda\xrightarrow{\mathcal F\text{-}o}0
\quad\text{in }Z.
\]
Thus
\[
\Phi(z_\lambda)=Tz_\lambda
\xrightarrow{\mathcal G\text{-}o}0.
\]

Since $Z$ is majorizing in $X$, for every $x\in X_+$ there exists
$w\in Z_+$ such that $x\leq w$. Hence, if $z\in Z$ and
$0\leq z\leq x$, then
\[
0\leq Tz\leq Tw,
\]
because $T$ is positive. Therefore,
\[
\{Tz:z\in Z,\ 0\leq z\leq x\}
\]
is order bounded in $Y$.

It now follows from Theorem~\ref{thm:supremum-envelope} that
\[
\widehat T(x_\alpha)
\xrightarrow{\mathcal G\text{-}o}0
\]
whenever
\[
0\leq x_\alpha\xrightarrow{\mathcal F\text{-}o}0
\quad\text{in }X.
\]
Hence $\widehat T$ is $\mathcal{FG}$-order continuous.

The uniqueness follows from
Theorem~\ref{thm:FG-continuous-implies-order-continuous}
and Veksler's extension theorem.
\end{proof}

Part \textup{(i)} of the following corollary recovers
\cite[Proposition~5]{TuranGurkok24}.

\begin{corollary}\label{cor:extension_special_cases}
Let $Z$ be an order dense majorizing Riesz subspace of a Riesz
space $X$, and let $Y$ be a Dedekind complete Riesz space. Then the
following statements hold.
\begin{enumerate}
\item[(i)] Every positive $uo$-continuous operator $T:Z\to Y$
has a unique positive $uo$-continuous extension to $X$.

\item[(ii)] Every positive strongly order continuous operator
$T:Z\to Y$ has a unique positive strongly order continuous
extension to $X$.

\item[(iii)] Suppose that $X$ and $Y$ are $f$-algebras, $Z$ is an
$f$-subalgebra of $X$, and $Y$ is semiprime. Then every positive
$mo$-continuous operator $T:Z\to Y$ has a unique positive
$mo$-continuous extension to $X$.
\end{enumerate}
\end{corollary}

\begin{proof}
For (i) and (ii), for each $z\in Z_+$ define
\[
F_z(x)=x\wedge z,\qquad x\in X_+,
\]
and set $\mathcal F=\mathcal T_Z$. Since $Z$ is majorizing in $X$,
$\mathcal T_Z$-order convergence and $uo$-convergence coincide in
$X$. Taking $\mathcal G=\mathcal T_Y$ gives (i), while taking
$\mathcal G=\mathcal I_Y$ gives (ii), by
Theorem~\ref{thm:extension-theorem}.

For (iii), for each $z\in Z_+$ define $M_z(x)=zx$ and set
$\mathcal F=\mathcal M_Z$. Since $Z$ is majorizing in $X$,
$\mathcal M_Z$-order convergence and $mo$-convergence coincide in
$X$. Taking $\mathcal G=\mathcal M_Y$, the result follows from
Theorem~\ref{thm:extension-theorem}.
\end{proof}

Let $T:X\to Y$ be a positive operator between Riesz spaces, where
$Y$ is Dedekind complete. For every ideal $A$ of $X$, the formula
\[
T_A(x)=\sup\{T(u):u\in A\cap[0,x]\},
\qquad x\in X_+,
\]
defines a positive operator $T_A:X\to Y$.
The following result, analogous to
\cite[Theorem~1.64]{AliprantisBurkinshaw06}, characterizes when a
positive operator is $\mathcal{FG}$-order continuous on a given ideal.

\begin{theorem}\label{thm:ideal-operator-continuity}
Let $A$ be an ideal of a Riesz space $X$, and let $Y$ be a
Dedekind complete Riesz space. Let
$\mathcal F\subseteq\operatorname{IS}(X)$ be such that every
$F\in\mathcal F$ is bounded on $A_+$ and satisfies
\[
F(A_+)\subseteq A_+.
\]
Let
\[
\mathcal G\subseteq\operatorname{IS}_{\vee}(Y)
\]
have the kernel separation property.
Let $T:X\to Y$ be a positive operator. Then $T$ is
$\mathcal{FG}$-order continuous on $A$ if and only if
$T_A$ is $\mathcal{FG}$-order continuous on $X$.
\end{theorem}

\begin{proof}
Suppose first that $T_A$ is $\mathcal{FG}$-order continuous on $X$.
Let
\[
x_\alpha\xrightarrow{\mathcal F\text{-}o}0
\quad\text{in }A.
\]
Since every ideal is regular,
\cite[Theorem~15]{Tantrawan25} yields
\[
x_\alpha\xrightarrow{\mathcal F\text{-}o}0
\quad\text{in }X.
\]
As $T_A=T$ on $A$, it follows that
\[
Tx_\alpha=T_Ax_\alpha
\xrightarrow{\mathcal G\text{-}o}0.
\]
Thus $T$ is $\mathcal{FG}$-order continuous on $A$.

Conversely, suppose that $T$ is $\mathcal{FG}$-order continuous
on $A$, and define
\[
\Phi:A_+\to Y_+,
\qquad
\Phi(u)=T(u).
\]
Let $(z_\lambda)\subseteq A_+$ satisfy
\[
z_\lambda\xrightarrow{\mathcal F\text{-}o}0
\quad\text{in }X.
\]
Since every ideal is regular, every $F\in\mathcal F$ is bounded on
$A_+$, and $F(A_+)\subseteq A_+$,
\cite[Corollary~16]{Tantrawan25} yields
\[
z_\lambda\xrightarrow{\mathcal F\text{-}o}0
\quad\text{in }A.
\]
Hence
\[
\Phi(z_\lambda)=Tz_\lambda
\xrightarrow{\mathcal G\text{-}o}0.
\]

For every $x\in X_+$,
\[
0\leq T(u)\leq T(x)
\qquad
(u\in A\cap[0,x]),
\]
so
\[
\{\Phi(u):u\in A\cap[0,x]\}
\]
is order bounded in $Y$. By
Theorem~\ref{thm:supremum-envelope},
\[
T_A(x_\alpha)
=
\sup\{\Phi(u):u\in A\cap[0,x_\alpha]\}
\xrightarrow{\mathcal G\text{-}o}0
\]
whenever
\[
0\leq x_\alpha\xrightarrow{\mathcal F\text{-}o}0
\quad\text{in }X.
\]
Thus $T_A$ is $\mathcal{FG}$-order continuous on $X$.
\end{proof}

As immediate consequences of
Theorem~\ref{thm:ideal-operator-continuity}, we obtain the
corresponding results for $uo$-continuous and strongly order
continuous operators. The $uo$-continuous case recovers
\cite[Theorem~8]{TuranGurkok24}.

\section*{Declarations}

\paragraph{Funding}
The author received no funding for this work.

\paragraph{Competing interests}
The author declares no competing interests.

\paragraph{Data availability}
No datasets were generated or analyzed during the current study.

\paragraph{Author contributions}
The author is solely responsible for the conception, analysis,
and preparation of the manuscript.

\end{document}